\documentclass[12pt]{article}

\usepackage{amssymb}
\usepackage{amsmath,amsthm}
\usepackage[dvips]{graphicx}
\usepackage{wrapfig}
\usepackage{bm}
\usepackage{url}

\newtheorem{theorem}{Theorem}[section]

\newtheorem{lemma}[theorem]{Lemma}

\theoremstyle{definition}

\newtheorem{example}[theorem]{Example}
\theoremstyle{remark}
\newtheorem{remark}[theorem]{Remark.}

\title{A note on the cross-index of a complete graph based on a linear tree}

\author{Yusuke Gokan \thanks{Department of Information and Computer Engineering, National Institute of Technology, Gunma College, 580 Toriba-cho, Maebashi-shi, Gunma 371-8530, Japan. } 
\and 
Hayato Katsumata \thanks{Department of Information and Computer Engineering, National Institute of Technology, Gunma College, 580 Toriba-cho, Maebashi-shi, Gunma 371-8530, Japan. } 
\and 
Katsuya Nakajima \thanks{Department of Information and Computer Engineering, National Institute of Technology, Gunma College, 580 Toriba-cho, Maebashi-shi, Gunma 371-8530, Japan. } 
\and 
Ayaka Shimizu \thanks{Department of Mathematics, National Institute of Technology, Gunma College, 580 Toriba-cho, Maebashi-shi, Gunma 371-8530, Japan. Email: shimizu@nat.gunma-ct.ac.jp } 
\and 
Yoshiro Yaguchi \thanks{Department of Mathematics, National Institute of Technology, Gunma College, 580 Toriba-cho, Maebashi-shi, Gunma 371-8530, Japan. Email: yaguchi-y@nat.gunma-ct.ac.jp }
}
\date{\today}

\begin{document}

\maketitle

\begin{abstract}
In this paper it is shown that a complete graph with $n$ vertices has an optimal diagram, i.e., a diagram whose crossing number equals the value of Guy's formula, with a free maximal linear tree and without free hamiltonian cycles for any odd integer $n\geq 7$. 
\end{abstract}

\section{Introduction}

A {\it graph} is a pair of sets of vertices and edges. 
A {\it diagram} $D$ of a graph $G$ is a diagram on the sphere $S^2$ of a spatial embedding of $G$. 
Thus, crossings of a diagram are only transverse and double points on edge diagrams at their interiors, where an {\it edge diagram} means a part of a diagram of a graph corresponding to an edge. 
The {\it crossing number of a diagram $D$ of a graph}, $c(D)$, is the number of crossings of $D$, 
and the {\it crossing number of a graph $G$}, $c(G)$, is the minimal value of $c(D)$ over all diagrams $D$ of $G$. 
In this paper, the crossing number of a graph is compared with the crossing number of a diagram called a based diagram which is used to define a $\Gamma$-unknotted graph for spatial graphs in \cite{AK-tr}. Let $n$ be a positive integer. 
A {\it complete graph}, $K_n$, is a graph of $n$ vertices such that each pair of two vertices are adjacent by a single edge, and each edge has different vertices on the endpoints. 
Let 
\begin{align*}
Z(n)=\frac{1}{4} \left\lfloor \frac{n}{2} \right\rfloor \left\lfloor \frac{n-1}{2} \right\rfloor \left\lfloor \frac{n-2}{2} \right\rfloor \left\lfloor \frac{n-3}{2} \right\rfloor ,
\end{align*}
where $\lfloor \ \rfloor$ is the floor function. 
Guy and Hill conjectured that the equality $c(K_n)=Z(n)$ holds for each positive integer $n$ (\cite{RKGuy1, HH}). 
It is known that Guy's conjecture (or Hill's conjecture), $c(K_n)=Z(n)$, is true for $n \le 12$ (\cite{PR}), and it is unknown for $n \ge 13$. 
(For $n=13$, it is shown in \cite{MPR} that $c(K_{13})=219, 221, 223$ or $225$, where $225$ is the value of $Z(13)$.) 
It is also known that the inequality $c(K_n) \le Z(n)$ holds for any $n$ (see, for example, \cite{BK, RKGuy2, HH, RT}). We say a diagram $D$ of $K_n$ is {\it optimal} if $c(D)=Z(n)$. 

A {\it maximal tree} $T$ of a graph $G$ is a connected subgraph of $G$ which contains no cycle, and contains all the vertices of $G$. 
A pair of a graph $G$ and its maximal tree $T$ is referred to as a {\it based graph} $(G; T)$. 
As shown in \cite{KSY} (see also \cite{AK-tr}), for any based graph $(G; T)$, there is a diagram $D$ of $G$ such that there are no crossings at the part corresponding to the maximal tree $T$. 
Such a diagram is denoted by $(D; T)$ and called a {\it based diagram $(D; T)$ of a based graph $(G; T)$}. 
The part of a diagram corresponding to a tree $T$ is also denoted by $T$ in this paper. 
For a based diagram $(D; T)$ of a based graph $(G; T)$, take a sufficiently small regular neighborhood $N$ of $T$ in $S^2$ such that each edge diagram other than in $T$ has just two intersection points with the boundary $C$ of $N$, i.e., $C$ is a simple closed curve in $S^2$. 
For two edge diagrams $e$ and $f$ which are not contained in $T$, let $e^1, e^2$ (resp. $f^1, f^2$) be the intersection points of $e$ (resp. $f$) and $C$. 
Now the {\it cross-index between $e$ and $f$}, $\varepsilon _T (e, f)$, is defined as follows (defined in \cite{KSY}): 
$\varepsilon _T (e, f)=0$ if there are the intersection points $e^{\alpha}, f^{\beta}, f^{\gamma}, e^{\delta}$ ($\alpha , \beta , \gamma , \delta \in \{ 1, 2 \} $) on $C$ in this cyclic order, and $\varepsilon _T (e, f)=1$ if there are the intersection points $e^{\alpha}, f^{\beta}, e^{\gamma}, f^{\delta}$ ($\alpha , \beta , \gamma , \delta \in \{ 1, 2 \} $) on $C$ in this cyclic order. (See Figure \ref{ef}.)
\begin{figure}[ht]
\begin{center}
\includegraphics[width=80mm]{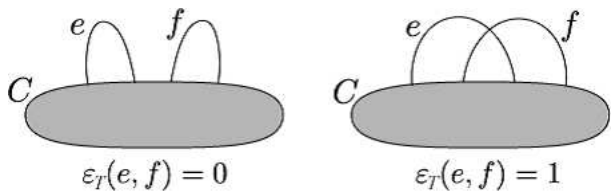}
\caption{The cross-index between edge diagrams in a based diagram. }
\label{ef}
\end{center}
\end{figure}
Let $e_1, e_2, \dots ,e_s$ be all the edge diagrams in a based diagram $(D; T)$ which are not in $T$. 
The {\it cross-index of a based diagram $(D; T)$} is defined as follows: 
\begin{align*}
\varepsilon (D; T) = \sum _{1\le i < j \le s} \varepsilon _T (e_i, e_j).
\end{align*}
The {\it cross-index $\varepsilon (G; T)$ of a based graph $(G; T)$} is the minimal value of $\varepsilon (D; T)$ for all based diagrams $(D; T)$ of $(G; T)$. 
The {\it cross-index $\varepsilon (G)$ of a graph $G$} is the minimal value of $\varepsilon (G; T)$ for all maximal trees $T$ of $G$. 
Let $c(D; T)$ denote the crossing number of a based diagram $(D; T)$ and let $c(G; T)$ be the minimal value of $c(D; T)$ for all based diagrams $(D; T)$ of $(G; T)$.  
The cross-index $\varepsilon (G; T)$ is equal to $c(G; T)$. (``Calculation Lemma'' in \cite{KSY}). 
Hence the inequality $c(G) \le \varepsilon (G) \le \varepsilon(G; T) = c(G; T)$ holds for any maximal tree $T$ of $G$ (Corollary 2.6 in \cite{KSY}). 
For a complete graph $K_n$, let $T^L$ be a maximal tree of $K_n$ which is linear, i.e., a maximal tree of $K_n$ without a vertex of degree three or more. 
The following question is asking whether the cross-index is equal to the crossing number, and whether a linear tree is a best possible tree on complete graphs: 

\phantom{x}
\noindent {\bf Question} (\cite{KSY}). \ Does the equality $c(K_n)= \varepsilon (K_n) = \varepsilon (K_n; T^L)$ hold?\\
\phantom{x}

\noindent It is confirmed that the equality holds for $n \le 12$ in \cite{KSY}. 
Hence it holds that $\varepsilon(K_n; T^L )=Z(n)$ for $n \le 12$. 
Let $H$ be a {\it hamiltonian cycle of} $K_n$, namely a cycle in $K_n$ containing all the vertices of $K_n$. 
Similarly to $\varepsilon(K_n; T)$, the cross-index $\varepsilon(K_n; H)$ of $K_n$ based on $H$ is defined. 
(Precise definition of $\varepsilon(K_n; H)$ is given in Section 2.) 
Since the hamiltonian cycle $H$ contains a linear tree $T^L$, the inequality $\varepsilon (D; T^L) \le \varepsilon (D; H)$ holds. 
Hence the inequality $c(K_n) \le \varepsilon (K_n) \le \varepsilon(K_n; T^L) \le \varepsilon (K_n ; H)$ holds.
In \cite{BK}, the inequality $\varepsilon (K_n ; H) \le Z(n)$ was shown by constructing a based diagram $(D; H)$ 
which is optimal, i.e., $\varepsilon (D; H)=Z(n)$, as follows: Draw a hamiltonian cycle $H$ as a regular $n$-gon, and draw all diagonals with positive slope (as straight line segments) and all other edges outside of $H$. 
Moreover, in \cite{AAFRS1}, it is shown that for any positive integer $n$, the equality $\varepsilon (K_n ; H)=Z(n)$ holds, and a sufficient condition for based diagrams $(D; H)$ of $(K_n; H)$ to be optimal is also described by using a {\it matrix representation}
which will be introduced in Section 2.

Besides, the based diagram $(D; T^L)$ of $K_7$ in Figure \ref{K7} satisfies $\varepsilon (D; T^L)=9=Z(7)$ and has no free hamiltonian cycle, where {\it a free hamiltonian cycle} of a diagram $D$ means a hamiltonian cycle which has no crossings. 
This paper shows the existence of an optimal based diagram $(D; T^L)$ of $(K_n; T^L)$ without a free hamiltonian cycle for an odd integer $n\geq 7$:
\begin{figure}[ht]
\begin{center}
\includegraphics[width=50mm]{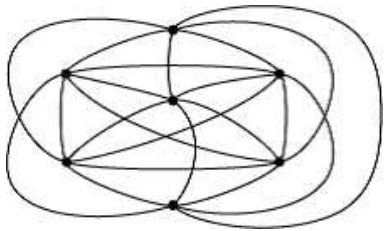}
\caption{An optimal based diagram $(D; T^L)$ of $(K_7; T^L)$ without free hamiltonian cycle.(The thick line represents a based linear tree.) }
\label{K7}
\end{center}
\end{figure}

\phantom{x}
\begin{theorem}
For any odd integer $n\geq 7$, $K_n$ has an optimal linear-tree based diagram $(D; T^L)$ such that there are no free hamiltonian cycles in $D$.
\label{main-thm}
\end{theorem}
\phantom{x}

\noindent Fix a positive integer $n$. For a diagram $D$ of $K_n$, let $V(D)$ be the set of all the vertices of $D$ and let $E(D)$ be the set of all the edge diagrams in $D$. 
Two diagrams $D$ and $D'$ of $K_n$ are {\it isomorphic} if there is a bijection from $V(D)$ to $V(D')$ that induces a bijection from $E(D)$ to $E(D')$ which sends each pair $(e_1,e_2)$ of $E(D)$ such that $e_1$ and $e_2$ crosses at their interiors to such a pair $(e_1',e_2')$ of $E(D')$. 
It is known that any optimal diagram $D$ of $K_n$ is isomorphic to the diagram shown in Figure \ref{K5K6} when $n=5, 6$. 
Hence, any optimal diagram $D$ of $K_n$ must have a free hamiltonian cycle when $n=5,6$. 
(See Figure \ref{K5K6}.)
\begin{figure}[ht]
\begin{center}
\includegraphics[width=90mm]{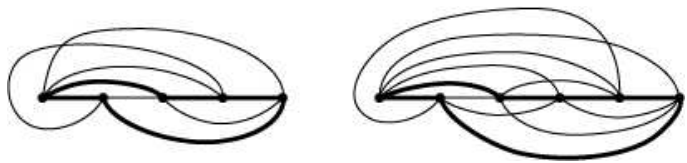}
\caption{Optimal diagrams of $K_5$ (left) and $K_6$ (right). (The thick lines represent hamiltonian cycles.) }
\label{K5K6}
\end{center}
\end{figure}

\noindent When $n=8$, $K_8$ has an optimal linear-tree based diagram $(D; T^L)$ such that there are no free hamiltonian cycles in $D$ as shown in Figure \ref{K8}. It is unknown if there exists such a based diagram for any even number $n\geq 10$.
\begin{figure}[ht]
\begin{center}
\includegraphics[width=50mm]{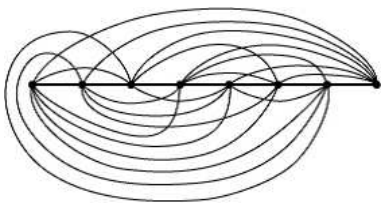}
\caption{An optimal based diagram $(D; T^L)$ of $(K_8; T^L)$ without free hamiltonian cycles. 
(The thick line represents $T^L$.) }
\label{K8}
\end{center}
\end{figure}

\noindent In \cite{AAFRS2}, a condition ``shellable'' for a diagram $D$ of $K_n$ is introduced, and it is proved that if $D$ is shellable then $c(D)\geq Z(n)$. In \cite{AAFRS2} it is also proved that any based diagram of $(K_n; H)$ is shellable. 
Recently, in \cite{AAFM}, a condition ``bishellable'', a generalization of ``shellable'', for a diagram $D$ of $K_n$ is introduced, and it is proved that if $D$ is bishellable then $c(D)\geq Z(n)$. It is also shown that any linear-tree based diagram of $(K_n; T^{L})$ is bishellable. Therefore, $\varepsilon (K_n; T^{L})=\varepsilon (K_n; H)=Z(n)$ holds. Let $\overline{\varepsilon}(K_n:; T^{L})$ denote the minimal number of crossing numbers for all linear-tree based diagrams of $(K_n; T^{L})$ which have no free hamiltonian cycle. The following theorem follows from Theorem \ref{main-thm}:

\phantom{x}
\begin{theorem}For any odd integer $n\geq 7$, we have $\overline{\varepsilon}(K_n; T^{L})=Z(n)$. 
\label{2nd-main-thm}
\end{theorem}
\phantom{x}

In Section 2, we will give a concrete matrix $M_n$ and prove that a based diagram $(D; H)$ of $(K_n; H)$ corresponding to $M_n$ is optimal by direct calculations, although $(D; H)$ is isomorphic to a diagram corresponding to a matrix introduced in [\cite{AAFRS1}, Theorem 19]. 
In Section 3, Theorem \ref{main-thm} is proved by using the matrix $M_n$. 

\section{A complete graph based on a hamiltonian cycle} 

Let $v_1, v_2, \dots ,v_n$ be the vertices of a complete graph $K_n$. 
Take a hamiltonian cycle $H=v_1v_2\dots v_nv_1$ of $K_n$ with the cyclic order $v_1, v_2, \dots , v_n, v_1$ without loss of generality. 
In this section, we consider a based diagram $(D; H)$ of $(K_n; H)$ on $S^2$ such that the hamiltonian cycle $H$ is on the equator and the other edges than $H$ are on the Northern or Southern Hemisphere. 
For the above based diagram $(D; H)$ of $(K_n; H)$, the {\it cross-index} is defined as follows: 
Let $e, f$ be two edge diagrams in $(D; H)$ which are not contained in $H$. 
Let $A$ be the boundary of a sufficiently small regular neighborhood of $H$ in $S^2$. 
Let $e^1, e^2$ (resp. $f^1, f^2$) be intersection points of $e$ (resp $f$) and $A$. 
The {\it cross-index between $e$ and $f$}, $\varepsilon _H (e,f)$, is defined to be $1$ if $e$ and $f$ are on the same side of $H$, and there are the intersection points $e^{\alpha}, f^{\beta}, e^{\gamma}, f^{\delta}$ ($\alpha , \beta , \gamma , \delta \in \{ 1,2 \}$) on $A$ in this cyclic order. 
Otherwise, let $\varepsilon _H (e,f)=0$. (See Figure \ref{efh}.) 
\begin{figure}[ht]
\begin{center}
\includegraphics[width=100mm]{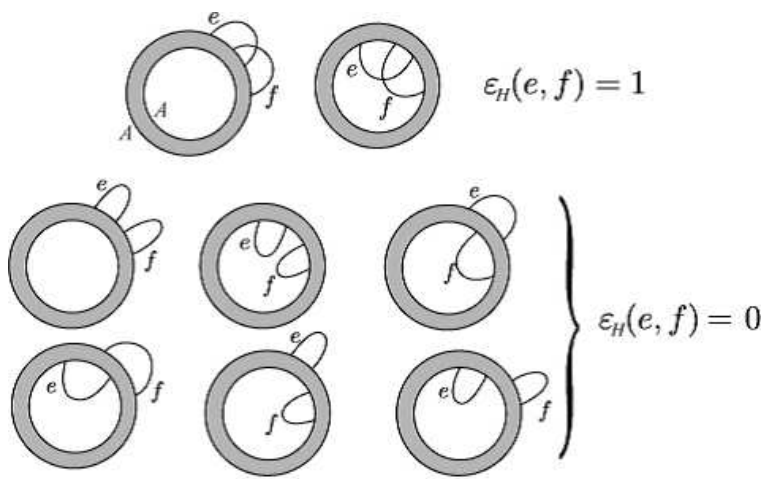}
\caption{The cross-index between two edge diagrams in a hamiltonian-cycle based diagram. }
\label{efh}
\end{center}
\end{figure}
The {\it cross-index} of a based diagram $(D; H)$ of $(K_n; H)$, denoted by $\varepsilon (D; H)$, is the sum of cross-indices for all the pairs of edge diagrams in $(D; H)$ which are not contained in $H$. 
Note that the cross-index $\varepsilon (D; H)$ depends only on whether each edge diagram is on the Northern Hemisphere or the Southern Hemisphere, i.e., depends only on the information around the neighborhood of $H$, not necessarily the whole $D$. 
Note that there are $2^{\frac{(n-3)n}{2}}$ possibilities of the choice of the hemispheres for each $n$. 
The {\it cross-index} of $(K_n; H)$, denoted by $\varepsilon (K_n; H)$, is the minimal value of $\varepsilon (D; H)$ for all based diagrams $(D; H)$. 
Let $c(K_n; H)$ denote the minimal number of crossings for all based diagrams $(D; H)$ of $(K_n ; H)$. 
Since the Northern and Southern Hemispheres bounded by $A$ are discs, similarly to the Calculation Lemma in \cite{KSY}, the equality $\varepsilon (K_n; H) = c(K_n ; H)$ holds. \\

Let $e_{(i,j)}$ denote the edge diagram in $(D; H)$ of $(K_n ; H)$ which has the vertices $v_i$ and $v_j$ at the endpoints. 
For each based diagram $(D; H)$ of $(K_n; H)$, define a matrix $M(D; H)=(a_{(i,j)})$ to be: 
\begin{align*}
a_{(i,j)}= \left\{ 
\begin{array}{ll}
1 & \text{ if } j \ge i+2 \text{ and } e_{(i,j)} \text{ is on the Northern Hemisphere} \\
0 & \text{ if } j \le i+1 \text{ or } (i,j)=(1,n) \\
-1 & \text{ if } j \ge i+2 \text{ and } e_{(i,j)} \text{ is on the Southern Hemisphere}
\end{array}\right.
\end{align*}
An example is shown in Figure \ref{ex-k6}. 
\begin{figure}[ht]
\begin{center}
\includegraphics[width=100mm]{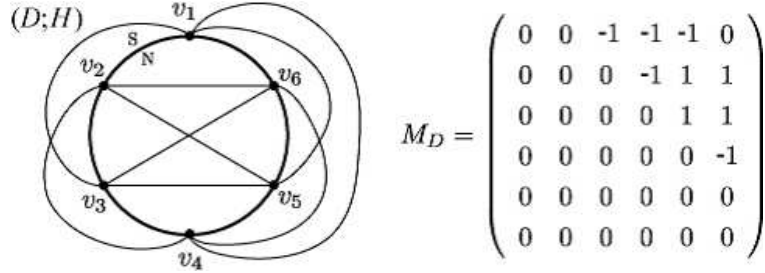}
\caption{A based diagram $(D;H)$ of $(K_6; H)$ and its matrix $M(D; H)$. }
\label{ex-k6}
\end{center}
\end{figure}
By definition, the following lemma holds: 

\phantom{x}
\begin{lemma}
Let $m_n=(a_{(i,j)})$ be an $n \times n$ matrix satisfying: 
\begin{align*}
a_{(i,j)}= \left\{ 
\begin{array}{ll}
0 & \text{ if } j \le i+1 \text{ or } (i,j)=(1,n) \\
1 \text { or } -1 & \text{ otherwise}
\end{array}\right. 
\end{align*}
Then there exists a based diagram $(D; H)$ of a based complete graph $(K_n ; H)$ such that $M(D; H)=m_n$ (not uniquely). 
\label{lem-matrix}
\end{lemma}
\phantom{x}

\noindent As shown in the following lemma, the cross-index can be obtained from the matrix. 

\phantom{x}
\begin{lemma}
Let $(D;H)$ be a based diagram of a based complete graph $(K_n; H)$ with the hamiltonian cycle $H=v_1 v_2 \dots v_n v_1$, and let $M(D; H) =(a_{(i,j)})$ be its matrix. 
Then the following formula holds:  
\begin{align*}
\varepsilon (D;H) = \sum _{i<k<j<l} \left\lfloor \frac{a_{(i,j)}\times a_{(k,l)} +1}{2} \right\rfloor.
\end{align*}
\label{ci-formula}
\end{lemma}

\begin{proof}
Let $e_{(i,j)}$ and $e_{(k,l)}$ be edge diagrams in $(D; H)$ which do not belong to $H$, where $i+1<j$, $k+1<l$ and $i<k$. 
By definition, $\varepsilon (e_{(i,j)}, e_{(k,l)})=1$ if and only if $i<k<j<l$ and $e_{(i,j)}$ and $e_{(k,l)}$ are on the same side of the hamiltonian cycle. 
In terms of matrices, $\varepsilon (e_{(i,j)}, e_{(k,l)})=1$ if and only if $i<k<j<l$ and the elements $a_{(i,j)}$ and $a_{(k,l)}$, which take $1$ or $-1$, have the same value, that is, $i<k<j<l$ and $\left\lfloor \frac{a_{(i,j)} \times a_{(k,l)} +1}{2}\right\rfloor =1$. 
Hence the cross-index $\varepsilon (D; H)$ is the sum of $\varepsilon (e_{(i,j)}, e_{(k,l)})=\left\lfloor \frac{a_{(i,j)} \times a_{(k,l)} +1}{2} \right\rfloor$ for all the pairs of edge diagrams $e_{(i,j)}$ and $e_{(k,l)}$ for $i<k<j<l$. 
\end{proof}
\phantom{x}
Let $K_n$ be a complete graph with vertices $v_1, v_2, \dots $ and $v_n$, 
and let $H$ be the hamiltonian cycle $v_1 v_2 \dots v_n v_1$. Let $M_n=(a_{(i,j)})$ be the $n\times n$ matrix defined by: 
$$
a_{(i,j)}= \begin{cases}
0 & \text{if } j\le i+1 \text{ or } (i,j)=(1,n) \\
1 & \text{if } j \ge i+2 \text{ and } \frac{n+2}{2}-i \le j \le n-i \\
    & \  \ \text{  or } j \ge i+2 \text{ and  } j \ge \frac{3n+2}{2}-i \\
-1 & \text{ otherwise}
\end{cases}\ \text{if}\ n \equiv 0 \pmod 2
$$ and 
$$
a_{(i,j)}= \begin{cases}
0 & \text{if } j\le i+1 \text{ or } (i,j)=(1,n) \\
1 & \text{if } j \ge i+2 \text{ and } \frac{n+3}{2}-i \le j \le n-i \\
    & \  \ \text{  or } j \ge i+2 \text{ and  } j \ge \frac{3n+1}{2}-i \\
-1 & \text{ otherwise}
\end{cases}
\ \text{if}\ n \equiv 1 \pmod 2.$$ Then, the following lemma holds: 
\phantom{x}
\begin{lemma}
Let $(D;H)$ be a based diagram of the based complete graph $(K_n ; H)$. If $M(D; H)=M_n$, then $(D; H)$ is optimal, i.e., 
$\varepsilon (D; H)=Z(n)$.
\label{mz-lem}
\end{lemma}
\phantom{x}
For a component $a_{(i,j)}$ of $M_n$ satisfying $a_{(i,j)} \ne 0$, 
let $\sigma (a_{(i,j)})$ be the number of components $a_{(k,l)}$ of $M_n$ satisfying $i<k<j<l$ and $a_{(k,l)}=a_{(i,j)}$, 
that is, $i<k<j<l$ and $\left\lfloor \frac{a_{(i,j)} \times a_{(k,l)} +1}{2}\right\rfloor =1$. 
Take a based diagram $(D,H)$ of a based graph $(K_n,H)$ such that $M(D;H)=M_n$. 
Since $\varepsilon (D;H)=\sum _{i<k<j<l} \left\lfloor \frac{a_{(i,j)}\times a_{(k,l)} +1}{2}\right\rfloor$ by Lemma \ref{ci-formula}, the cross-index is obtained by summing $\sigma (a_{(i,j)})$ for all the pairs of $i$ and $j$ satisfying $i+1<j$; that is, $\varepsilon (D;H)=\sum _{i+1<j} \sigma (a_{(i,j)})$.
As it can be seen from the following examples, the calculation $\sigma (a_{(i,j)})$ depends on the location of $a_{(i,j)}$. 
\begin{example}
For $n=14$, $\sigma (a_{(1,7)})$ is obtained as follows (see Figure \ref{a17}): 
\begin{align*}
\sigma (a_{(1,7)})=\sum _{k=2}^{6} \sum _{l=8} ^{14-k} 1 =15.
\end{align*}
\begin{figure}[ht]
\begin{center}
\includegraphics[width=65mm]{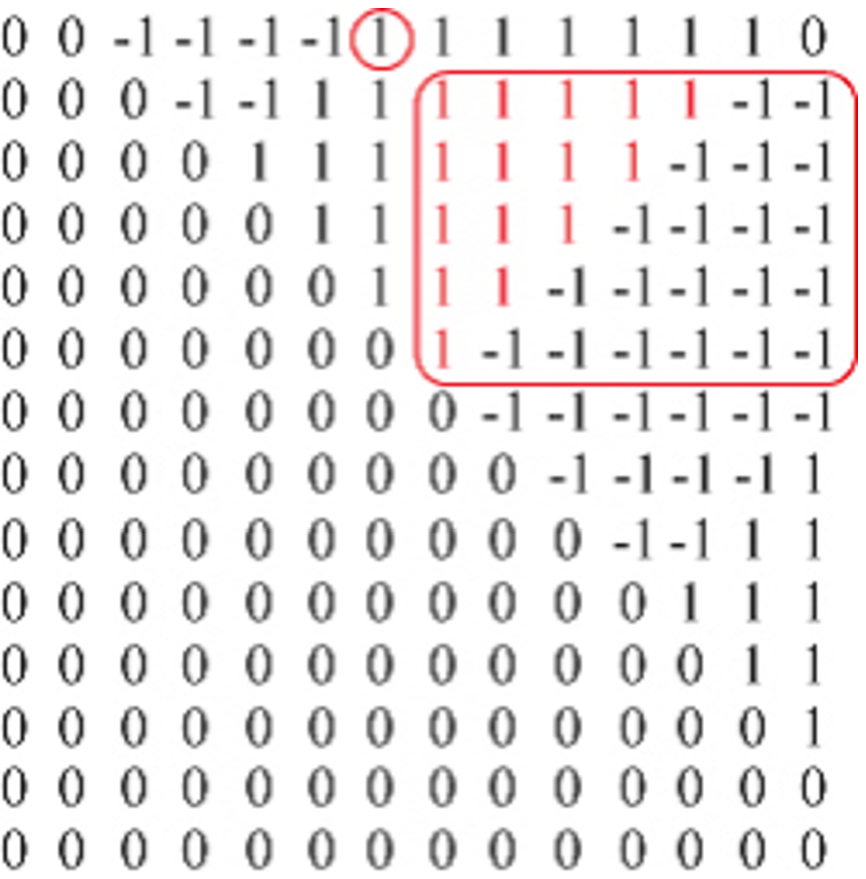}
\caption{The components $a_{(k,l)}$ of $M_n$ satisfying $1<k<7<l$ and $a_{(k,l)}=a_{(1,7)}=1$ are $a_{(2,8)}$, $a_{(2,9)}$, $a_{(2,10)}$, $a_{(2,11)}$, $a_{(2,12)}$, $a_{(3,8)}$, $a_{(3,9)}$, $a_{(3,10)}$, $a_{(3,11)}$, $a_{(4,8)}$, $a_{(4,9)}$, $a_{(4,10)}$, $a_{(5,8)}$, $a_{(5,9)}$ and $a_{(6,8)}$.}
\label{a17}
\end{center}
\end{figure}
\end{example}
\begin{example}
For $n=14$, $\sigma (a_{(2,10)})$ is obtained as follows (see Figure \ref{a210}): 
\begin{align*}
\sigma (a_{(2,10)})=\sum _{l=11}^{11} \sum _{k=3} ^{3} 1 +\sum _{l=13}^{14} \sum _{k=22-l} ^{9} 1 =4.
\end{align*}
\begin{figure}[ht]
\begin{center}
\includegraphics[width=65mm]{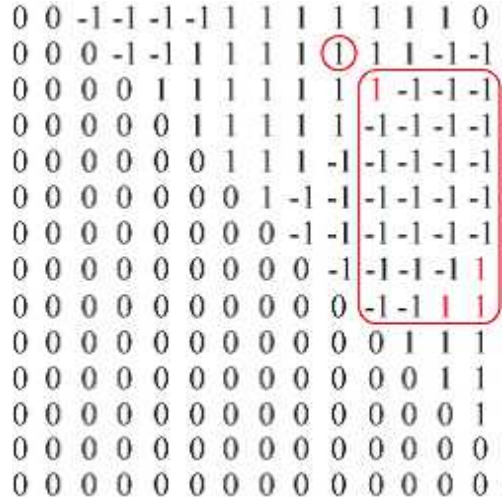}
\caption{The components $a_{(k,l)}$ of $M_n$ satisfying $2<k<10<l$ and $a_{(k,l)}=a_{(2,10)}=1$ are $a_{(3,11)}$, $a_{(8,14)}$, $a_{(9,13)}$ and $a_{(9,14)}$.}
\label{a210}
\end{center}
\end{figure}
\end{example}
\phantom{x}

\noindent We show Lemma \ref{mz-lem}.


\phantom{x}
\noindent {\bf Proof of Lemma \ref{mz-lem}.} 
Let $X=\{1,2,\dots ,n\}$ and let $\displaystyle X^2=X\times X$. Let $\displaystyle N=\{(i,j)\in X^2\ |\ a_{(i,j)}=1\}$ and let $\displaystyle S=\{(i,j)\in X^2\ |\ a_{(i,j)}=-1\}$.\\
\ (i) When $n \equiv 0 \pmod 4$: Let 
\begin{align*}
N_1 = & \left\{ (i,j)\in X^2 \ | \  \frac{n+2}{2}-j \le i \le j-2, \ \frac{n+8}{4} \le j \le \frac{n}{2} \right\} , \\
N_2 = & \left\{ (i,j)\in X^2 \ | \ 1 \le i \le n-j, \ \frac{n+2}{2} \le j \le \frac{3n}{4} \right\} , \\
N_3 = & \left\{ (i,j)\in X^2 \ | \ 1 \le i \le n-j, \ \frac{3n+4}{4} \le j \le n-1 \right\} , \\
N_4 = & \left\{ (i,j)\in X^2 \ | \ \frac{3n+2}{2}-j \le i \le j-2, \ \frac{3n+8}{4} \le j \le n \right\} .
\end{align*}
Then $N=N_1 \cup N_2 \cup N_3 \cup N_4$ and $N_s \cap N_t = \emptyset$ for distinct $s, t\in \{1,2,3,4\}$. Let  
\begin{align*}
S_1 = & \left\{ (i,j)\in X^2 \ | \ 1 \le i \le j-2, \ 3 \le j \le \frac{n}{4} \right\} , \\
S_2 = & \left\{ (i,j)\in X^2 \ | \ 1 \le i \le \frac{n}{2}-j, \ \frac{n+4}{4} \le j \le \frac{n-2}{2} \right\} , \\
S_3 = & \left\{ (i,j)\in X^2 \ | \ n-j+1 \le i \le \frac{n}{2}, \ \frac{n+4}{2} \le j \le \frac{3n}{4} \right\} , \\
S_4 = & \left\{ (i,j)\in X^2 \ | \ \frac{n+2}{2} \le i \le j-2, \ \frac{n+6}{2} \le j \le \frac{3n}{4} \right\} , \\
S_5 = & \left\{ (i,j)\in X^2 \ | \ n-j+1 \le i \le \frac{n}{2}, \ \frac{3n+4}{4} \le j \le n-1 \right\} , \\
S_6 = & \left\{ (i,j)\in X^2 \ | \ 2 \le i \le \frac{n}{2}, \ j=n \right\} , \\
S_7 = & \left\{ (i,j)\in X^2 \ | \ \frac{n+2}{2} \le i \le \frac{3n}{2}-j, \ \frac{3n+4}{4} \le j \le n-1 \right\} .
\end{align*}
Then $S=S_1 \cup S_2 \cup \dots \cup S_7$ and $S_s \cap S_t = \emptyset$ for distinct $s, t\in \{1,2,3,4,5,6,7\}$. 
(An example is shown in Figure \ref{m16}.) 
\begin{figure}[ht]
\begin{center}
\includegraphics[width=85mm]{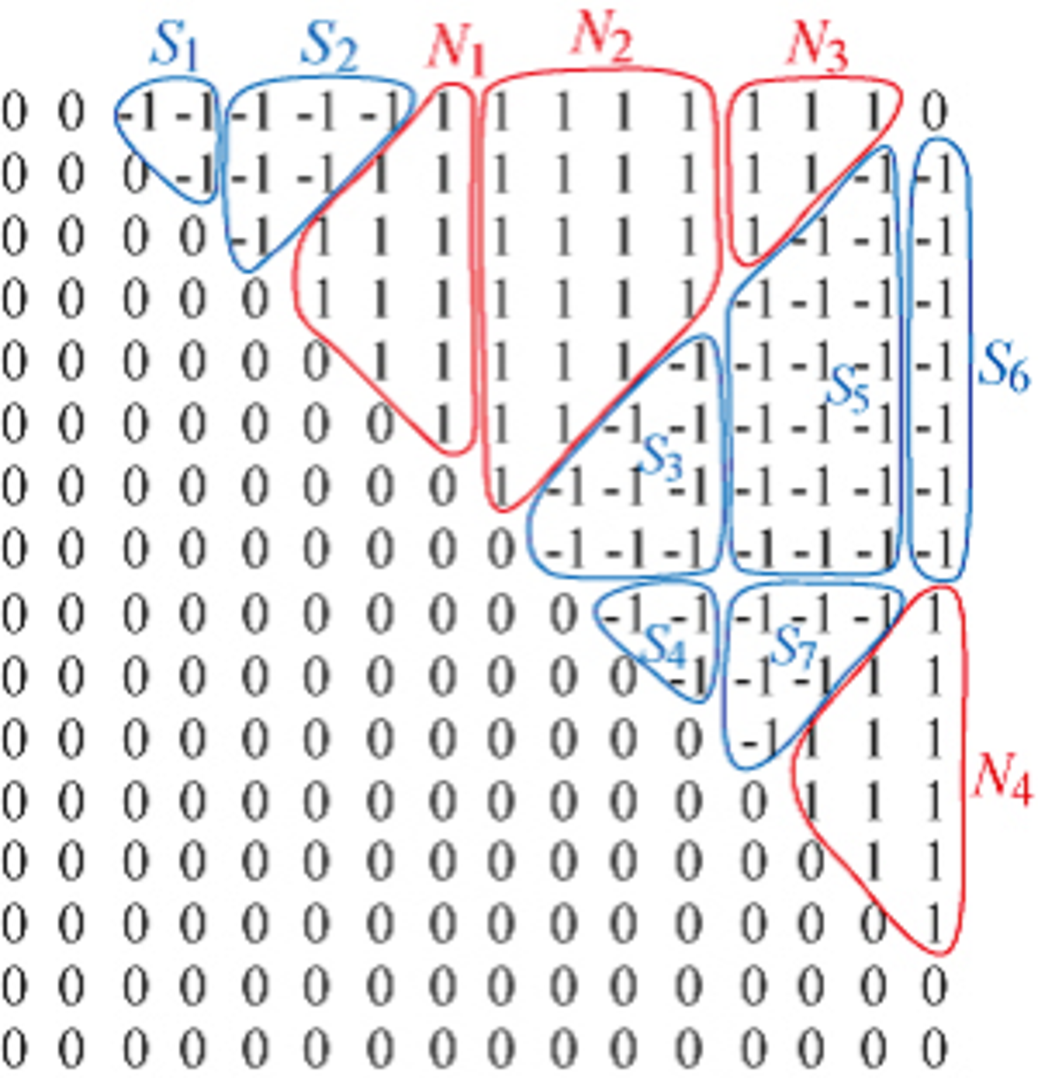}
\caption{Block division for $n=16$. }
\label{m16}
\end{center}
\end{figure}
The sum of $\sigma (a_{(i,j)})$ for all the elements $a_{(i,j)}$ which belong to $N_1$ is: 
\begin{align*}
\sum _{ (i,j)\in N_1} \sigma (a_{(i,j)}) = & \sum _{(i,j)\in N_1} \sum _{k=i+1} ^{j-1} \sum _{l=j+1} ^{n-k} 1\\
 = & \sum _{j=\frac{n+8}{4}} ^{\frac{n}{2}} \sum _{i=\frac{n+2}{2}-j} ^{j-2} \sum _{k=i+1} ^{j-1} \sum _{l=j+1} ^{n-k} 1   = \frac{n^2(n-2)(n-4)}{384}.
\end{align*}
\noindent Similarly, 
\begin{align*}
\sum _{(i,j)\in N_2} \sigma (a_{(i,j)}) = 
& \sum _{j=\frac{n+2}{2}} ^{\frac{3n}{4}} \sum _{i=1} ^{n-j} \left( \sum _{l=\frac{3n+4}{2}-j} ^{n} \sum _{k=\frac{3n+2}{2}-l} ^{j-1} 1 + \sum _{l=j+1} ^{n-i-1} \sum _{k=i+1} ^{n-l} 1 \right) 
=\frac{n(5n-4)(n-4)^2}{1536}, \\
\sum _{(i,j)\in N_3} \sigma (a_{(i,j)}) = 
& \sum _{j=\frac{3n+4}{4}} ^{n-1} \sum _{i=1} ^{n-j} \left( \sum _{l=j+1} ^{n-i-1} \sum _{k=i+1} ^{n-l} 1 + \sum _{l=j+1} ^{n} \sum _{k=\frac{3n+2}{2}-l} ^{j-1} 1 \right) 
=\frac{n(n-4)(n^2-4n+16)}{768}, \\
\sum _{(i,j)\in N_4} \sigma (a_{(i,j)}) = 
& \sum _{j=\frac{3n+8}{4}} ^{n} \sum _{i=\frac{3n+2}{2}-j} ^{j-2} \sum _{l=j+1} ^{n} \sum _{k=i+1} ^{j-1} 1 
=\frac{n^2(n-4)(n-8)}{1536}, 
\end{align*}
\begin{align*}
\sum _{(i,j)\in S_1} \sigma (a_{(i,j)}) = 
&  \sum _{j=3} ^{\frac{n}{4}} \sum _{i=1} ^{j-2} \left( \sum _{k=i+1} ^{j-1} \sum _{l=j+1} ^{\frac{n}{2}-k} 1+\sum_{k=i+1}^{j-1} \sum_{l=n+1-k}^{n} 1 \right) 
=\frac{n(n-4)(n-8)(5n-12)}{6144}, \\
\sum _{(i,j)\in S_2} \sigma (a_{(i,j)}) = 
&  \sum _{j=\frac{n+4}{4}} ^{\frac{n-2}{2}} \sum _{i=1} ^{\frac{n}{2}-j} \left(\sum _{k=i+1} ^{\frac{n-2}{2}-j} \sum _{l=j+1} ^{\frac{n}{2}-k} 1 + \sum _{k=i+1}^{j-1} \sum _{l=n+1-k}^{n} 1\right)
=\frac{n(n-4)(11n^2-44n+32)}{6144}, \\
\sum _{(i,j)\in S_3} \sigma (a_{(i,j)}) = 
& \sum _{j=\frac{n+4}{2}} ^{\frac{3n}{4}} \sum _{i=n-j+1} ^{\frac{n}{2}} \left(\sum _{l=j+1} ^{n} \sum _{k=i+1} ^{\frac{n}{2}} 1 + \sum _{k=\frac{n+2}{2}}^{j-1} \sum _{j+1}^{\frac{3n}{2}-k} 1\right)
=\frac{n(n+4)(n-4)(3n-8)}{1536}, \\
\sum _{ (i,j) \in S_4} \sigma (a_{(i,j)}) = 
& \sum _{j=\frac{n+6}{2}} ^{\frac{3n}{4}} \sum _{i=\frac{n+2}{2}} ^{j-2} \sum _{k=i+1} ^{j-1} \sum _{l=j+1} ^{\frac{3n}{2}-k} 1 
=\frac{n(n-4)(n-8)(3n-20)}{6144}, \\
\sum _{(i,j)\in S_5} \sigma (a_{(i,j)}) = 
& \sum _{j=\frac{3n+4}{4}} ^{n-1} \sum _{i=n-j+1} ^{\frac{n}{2}} \sum _{l=j+1} ^{n} \sum _{k=i+1} ^{\frac{3n}{2}-l} 1 
=\frac{n(n+1)(n-4)^2}{384}, \\
\sum _{(i,j)\in S_6} \sigma (a_{(i,j)}) = & 0, \\
\sum _{(i,j)\in S_7} \sigma (a_{(i,j)}) = 
& \sum _{j=\frac{3n+4}{4}} ^{n-1} \sum _{i=\frac{n+2}{2}} ^{\frac{3n}{2}-j} \sum _{l=j+1} ^{\frac{3n-2}{2}-i} \sum _{k=i+1} ^{\frac{3n}{2}-l} 1 
=\frac{n(n-4)(n-8)(n-12)}{6144}. \\
\end{align*}
Hence 
\begin{align*}
\varepsilon (D; H) = & \sum _{(i,j) \in N} \sigma (a_{(i,j)}) + \sum _{(i,j) \in S} \sigma (a_{(i,j)}) \\
 = & \sum_{t=1}^{4} \sum _{(i,j) \in N_t} \sigma (a_{(i,j)}) + \sum_{t=1}^{7} \sum _{(i,j) \in S_t} \sigma (a_{(i,j)}) \\
 = & \frac{n(n-4)(n-2)^2}{128} + \frac{n(n-4)(n-2)^2}{128} \\
 = & \frac{n(n-4)(n-2)^2}{64} = Z(n).
\end{align*}


\noindent  (ii) When $n \equiv 1 \pmod 4$: Let 
\begin{align*}
N_1 = & \left\{ (i,j)\in X^2 \ | \  \frac{n+3}{2}-j \le i \le j-2, \ \frac{n+7}{4} \le j \le \frac{n-1}{2}   \right\} , \\
N_2 = & \left\{ (i,j)\in X^2 \ | \  1 \le i \le \frac{n-3}{2}, \  j = \frac{n+1}{2}   \right\} , \\
N_3 = & \left\{ (i,j)\in X^2 \ | \  1 \le i \le n-j, \ \frac{n+5}{2} \le j \le \frac{3n+1}{4}  \right\} , \\
N_4 = & \left\{ (i,j)\in X^2 \ | \  1 \le i \le n-j, \ \frac{3n+5}{4} \le j \le n-1  \right\} , \\
N_5 = & \left\{ (i,j)\in X^2 \ | \  \frac{3n+1}{2}-j \le i \le j-2, \ \frac{3n+5}{4} \le j \le n  \right\} .
\end{align*}
Then $N=N_1 \cup N_2 \cup N_3 \cup N_4 \cup N_5$ and $N_s \cap N_t = \emptyset$ for distinct $s, t\in \{1,2,3,4,5\}$. Let 
\begin{align*}
S_1 = & \left\{ (i,j)\in X^2 \ | \  1 \le i \le j-2, \ 3 \le j \le \frac{n-1}{4}    \right\} , \\
S_2 = & \left\{ (i,j)\in X^2 \ | \  \frac{n+3}{4} \le j \le \frac{n+1}{2}-i, \ 1 \le i \le \frac{n-5}{4}   \right\} , \\
S_3 = & \left\{ (i,j)\in X^2 \ | \  n-j+1 \le i \le \frac{n-1}{2}, \ \frac{n+3}{2} \le j \le \frac{3n-3}{4}  \right\} , \\
S_4 = & \left\{ (i,j)\in X^2 \ | \  \frac{n+1}{2} \le i \le j-2, \ \frac{n+5}{2} \le j \le \frac{3n-3}{4}  \right\} , \\
S_5 = & \left\{ (i,j)\in X^2 \ | \  n-j+1 \le i \le \frac{n-1}{2}, \ \frac{3n+1}{4} \le j \le n-1   \right\}, \\
S_6 = & \left\{ (i,j)\in X^2 \ | \  2 \le i \le \frac{n-1}{2},\ j=n   \right\}, \\
S_7 = & \left\{ (i,j)\in X^2 \ | \  \frac{3n+1}{4} \le j \le \frac{3n-1}{2}-i, \ \frac{n+1}{2} \le i \le \frac{3n-7}{4}  \right\} .
\end{align*}
Then $S=S_1 \cup S_2 \cup \dots \cup S_7$ and $S_s \cap S_t = \emptyset$ for distinct $s, t\in \{1,2,3,4,5,6,7\}$. 
(An example is shown in Figure \ref{m17}.) 
\begin{figure}[ht]
\begin{center}
\includegraphics[width=85mm]{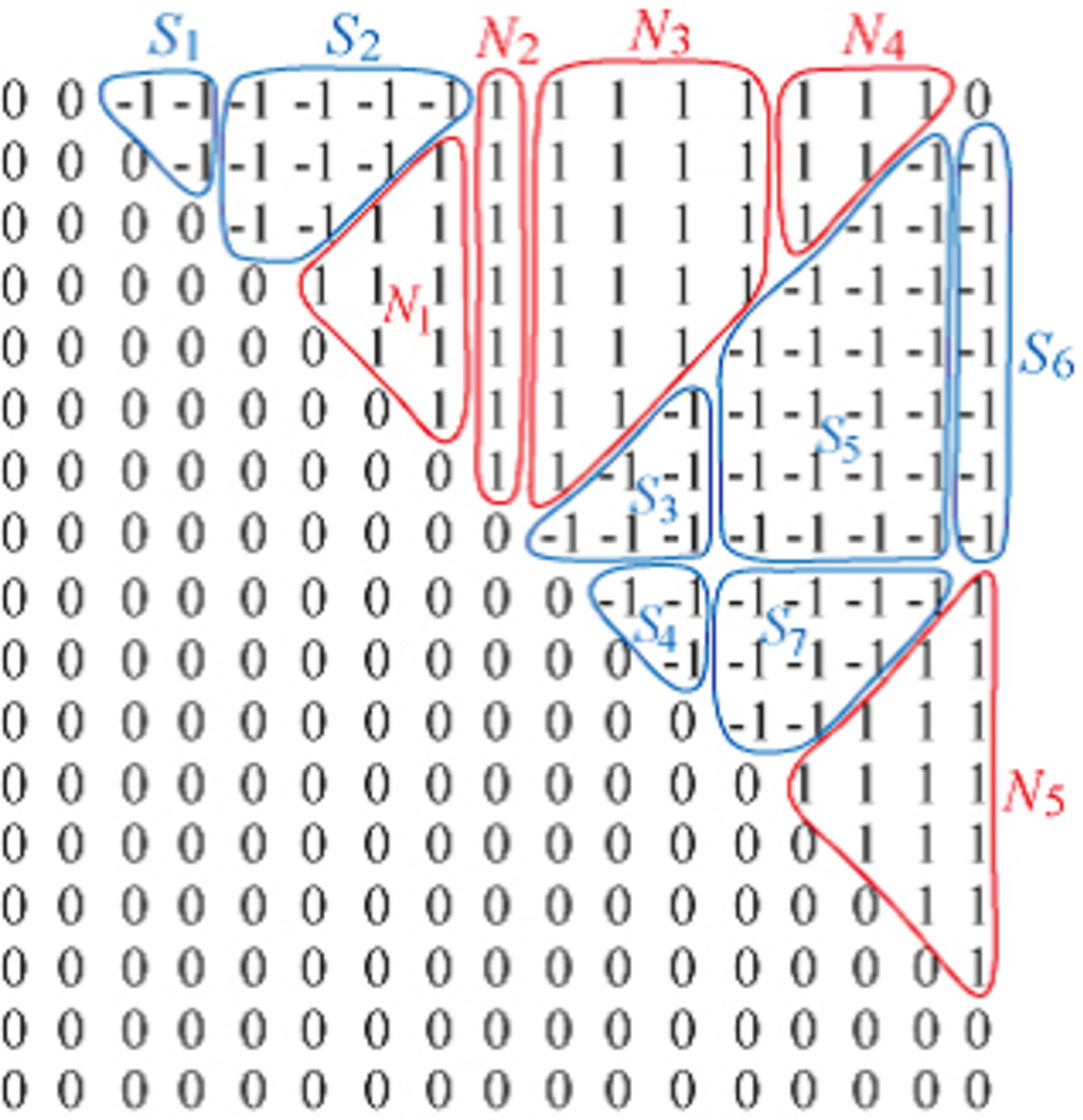}
\caption{Block division for $n=17$. }
\label{m17}
\end{center}
\end{figure}
\begin{align*}
\sum _{(i,j)\in N_1} \sigma (a_{(i,j)}) = & \sum _{j=\frac{n+7}{4}} ^{\frac{n-1}{2}} \sum _{i=\frac{n+3}{2}-j} ^{j-2} \sum _{k=i+1} ^{j-1} \sum _{l=j+1} ^{n-k} 1 
=\frac{(n-1)(n-3)(n-5)^2}{384}, \\
\sum _{(i,j)\in N_2} \sigma (a_{(i,j)}) = & \sum _{i=1} ^{\frac{n-3}{2}} \sum _{k=i+1} ^{\frac{n-3}{2}} \sum _{l=\frac{n+3}{2}} ^{n-k} 1 
=\frac{(n-1)(n-3)(n-5)}{48}, \\
\sum _{(i,j)\in N_3} \sigma (a_{(i,j)}) = & \sum _{j=\frac{n+3}{2}} ^{\frac{3n+1}{4}} \sum _{i=1} ^{n-j} \sum _{l=\frac{3n+3}{2}-j} ^{n} \sum _{k=\frac{3n+1}{2}-l} ^{j-1} 1 
=\frac{(n-1)(5n^3-43n^2+231n-321)}{1536}, \\
\sum _{(i,j)\in N_4} \sigma (a_{(i,j)}) = & \sum _{j=\frac{3n+5}{4}} ^{n-1} \sum _{i=1} ^{n-j} \sum _{l=j+1} ^{n} \sum _{k=\frac{3n+1}{2}-l} ^{j-1} 1 
=\frac{(n-1)(n-5)(n^2-2n+9)}{768}, \\
\sum _{(i,j)\in N_5} \sigma (a_{(i,j)}) = & \sum _{j=\frac{3n+5}{4}} ^{n} \sum _{i=\frac{3n+1}{2}-j} ^{j-2} \sum _{l=j+1} ^{n} \sum _{k=i+1} ^{j-1} 1 
=\frac{(n+3)(n-1)(n-5)^2}{1536}, 
\end{align*}
\begin{align*}
\sum _{(i,j) \in S_1} \sigma (a_{(i,j)}) = 
& \sum _{j=3} ^{\frac{n-1}{4}} \sum _{i=1} ^{j-2} \left( \sum _{k=i+1} ^{j-1} \sum _{l=j+1} ^{\frac{n+1}{2}-k} 1 + \sum _{k=i+1} ^{j-1} \sum _{l=n-k+1} ^{n} 1 \right) 
=\frac{(n-1) (n-5) (n-9) (5n-1)}{6144}, \\
\sum _{(i,j) \in S_2} \sigma (a_{(i,j)}) = 
& \sum _{i=1} ^{\frac{n-5}{4}} \sum _{j=\frac{n+3}{4}} ^{\frac{n+1}{2}-i} \left( \sum _{k=i+1} ^{\frac{n-1}{2}-j} \sum _{l=j+1} ^{\frac{n+1}{2}-k} 1+\sum _{k=i+1} ^{j-1} \sum _{l=n-k+1} ^{n} 1 \right) \\
= & \frac{(n+3)(n-1)(n-5)(11n+13)}{6144}, \\
\sum _{(i,j) \in S_3} \sigma (a_{(i,j)}) = 
& \sum _{j=\frac{n+3}{2}} ^{\frac{3n-3}{4}} \sum _{i=n-j+1} ^{\frac{n-1}{2}}  \left( \sum _{l=j+1} ^{n} \sum _{k=i+1} ^{\frac{n-1}{2}} 1+\sum _{k=\frac{n+1}{2}} ^{j-1} \sum _{l=j+1} ^{\frac{3n-1}{2}-k} 1 \right)  
=\frac{(n-1)^2(n-5)^2}{512}, \\
\sum _{(i,j) \in S_4} \sigma (a_{(i,j)}) = 
& \sum _{j=\frac{n+5}{2}} ^{\frac{3n-3}{4}} \sum _{i=\frac{n+1}{2}} ^{j-2} \sum _{k=i+1} ^{j-1} \sum _{l=j+1} ^{\frac{3n-1}{2}-k} 1 
=\frac{(n-1)(n-5)(n-9)(3n-7)}{6144}, \\
\sum _{(i,j) \in S_5} \sigma (a_{(i,j)}) = 
& \sum _{j=\frac{3n+1}{4}} ^{n-1} \sum _{i=n+1-j} ^{\frac{n-1}{2}} \sum _{l=j+1} ^{n} \sum _{k=i+1} ^{\frac{3n-1}{2}-l} 1 
=\frac{(n+3)(n-1)(n-2)(n-5)}{384}, \\
\sum _{ (i,j) \in S_6} \sigma (a_{(i,j)}) = & 0, \\
\sum _{(i,j) \in S_7} \sigma (a_{(i,j)}) = 
& \sum _{i=\frac{n+1}{2}} ^{\frac{3n-7}{4}} \sum _{j=\frac{3n+1}{4}} ^{\frac{3n-1}{2}-i} \sum _{l=j+1} ^{\frac{3n-3}{2}-i} \sum _{k=i+1} ^{\frac{3n-1}{2}-l} 1 
=\frac{(n+3)(n-1)(n-5)(n-9)}{6144}.
\end{align*}
Hence 
\begin{align*}
\varepsilon (D; H) = & \sum _{(i,j) \in N} \sigma (a_{(i,j)}) + \sum _{(i,j) \in S} \sigma (a_{(i,j)}) \\
 = & \sum_{t=1}^{5} \sum _{(i,j) \in N_t} \sigma (a_{(i,j)}) + \sum_{t=1}^{7} \sum _{(i,j) \in S_t} \sigma (a_{(i,j)}) \\
 = & \frac{(n-1)^2 (n^2-6n+13)}{128} + \frac{(n-5)(n-1)^3}{128} \\
 = & \frac{(n-1)^2(n-3)^2}{64} = Z(n).
\end{align*}


\noindent  (iii) When $n \equiv 2 \pmod 4$: Let
\begin{align*}
N_1 = & \left\{ (i,j)\in X^2 \ | \  \frac{n+2}{2}-j \le i \le j-2, \ \frac{n+6}{4} \le j \le \frac{n}{2}   \right\} , \\
N_2 = & \left\{ (i,j)\in X^2 \ | \  1 \le i \le n-j, \ \frac{n+2}{2} \le j \le \frac{3n+2}{4}  \right\} , \\
N_3 = & \left\{ (i,j)\in X^2 \ | \  1 \le i \le n-j, \ \frac{3n+6}{4} \le j \le n-1  \right\} , \\
N_4 = & \left\{ (i,j)\in X^2 \ | \  \frac{3n+2}{2}-j \le i \le j-2, \ \frac{3n+6}{4} \le j \le n  \right\} .
\end{align*}
Then $N=N_1 \cup N_2 \cup N_3 \cup N_4$ and $N_s \cap N_t = \emptyset$ for distinct $s, t\in \{1,2,3,4\}$. 
Let $S= \{ a_{(i,j)} \ | \ a_{(i,j)}=-1 \}$, and let 
\begin{align*}
S_1 = & \left\{ (i,j)\in X^2 \ | \  1 \le i \le j-2, \ 3 \le j \le \frac{n-2}{4}    \right\} , \\
S_2 = & \left\{ (i,j)\in X^2 \ | \  \frac{n+2}{4} \le j \le \frac{n}{2}-i, \ 1 \le i \le \frac{n-6}{4}   \right\} , \\
S_3 = & \left\{ (i,j)\in X^2 \ | \  n-j+1 \le i \le \frac{n}{2}, \ \frac{n+4}{2} \le j \le \frac{3n-2}{4}  \right\} , \\
S_4 = & \left\{ (i,j)\in X^2 \ | \  \frac{n+2}{2} \le i \le j-2, \ \frac{n+6}{2} \le j \le \frac{3n-2}{4}  \right\} , \\
S_5 = & \left\{ (i,j)\in X^2 \ | \  n-j+1 \le i \le \frac{n}{2}, \ \frac{3n+10}{4} \le j \le n-1   \right\}\\
S_6 = & \left\{ (i,j)\in X^2 \ | \  2 \le i \le \frac{n}{2},\ j=n   \right\}\\
S_7 = & \left\{ (i,j)\in X^2 \ | \  \frac{3n+2}{4} \le j \le \frac{3n}{2}-i, \ \frac{n+2}{2} \le i \le \frac{3n-6}{4}  \right\} .
\end{align*}
Then $S=S_1 \cup S_2 \cup \dots \cup S_7$ and $S_s \cap S_t = \emptyset$ for distinct $s, t\in \{1,2,3,4,5,6,7\}$. 
(An example is shown in Figure \ref{m18}.) 
\begin{figure}[ht]
\begin{center}
\includegraphics[width=90mm]{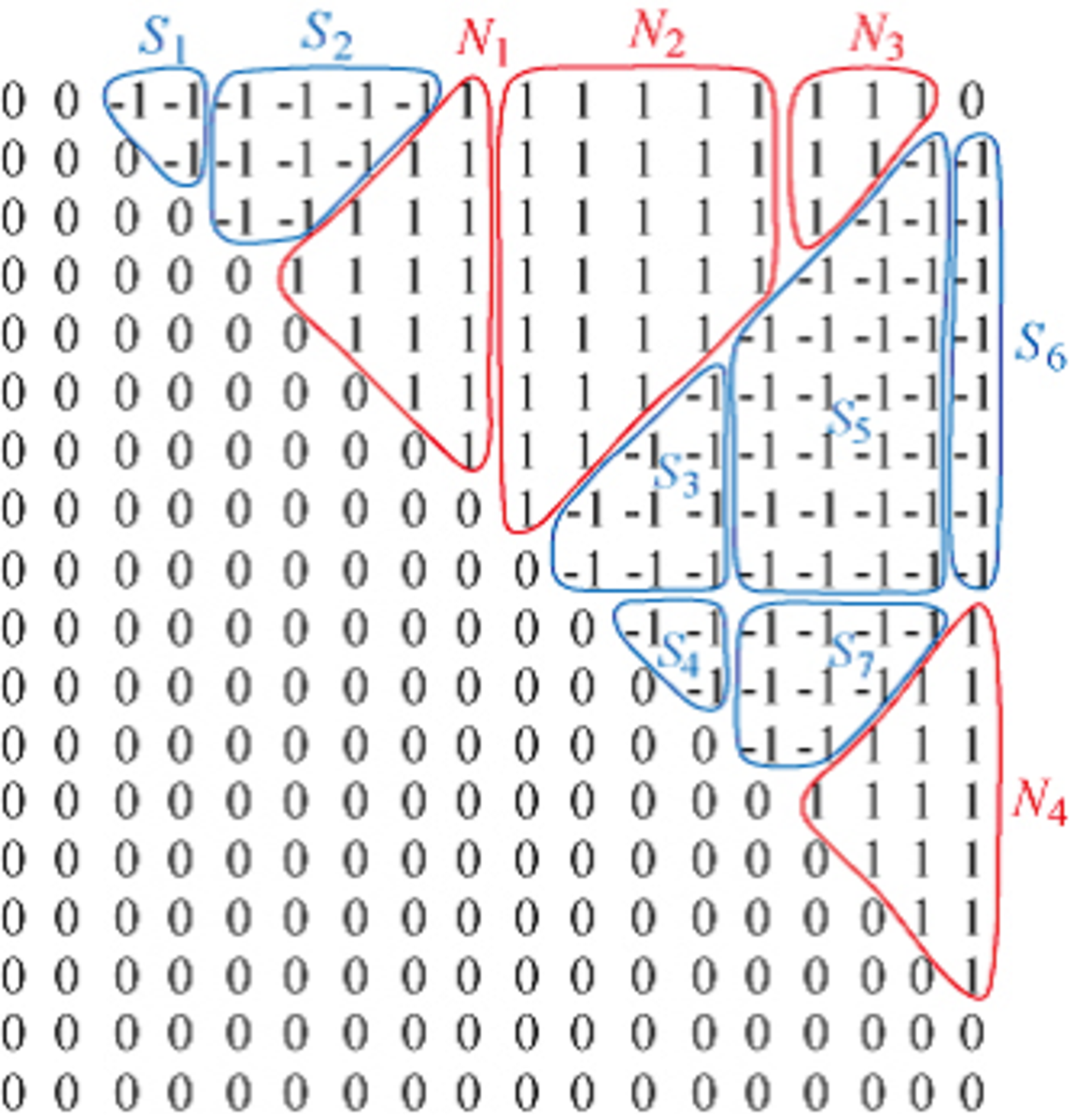}
\caption{Block division for $n=18$. }
\label{m18}
\end{center}
\end{figure}

\begin{align*}
\sum _{ (i,j) \in N_1} \sigma (a_{(i,j)}) = 
& \sum _{j=\frac{n+6}{4}} ^{\frac{n}{2}} \sum _{i=\frac{n+2}{2}-j} ^{j-2} \sum _{k=i+1} ^{j-1} \sum _{l=j+1} ^{n-k} 1 
=\frac{(n+2)(n-2)(n^2-6n+12)}{384}, \\
\sum _{ (i,j) \in N_2} \sigma (a_{(i,j)}) = 
& \sum _{j=\frac{n+2}{2}} ^{\frac{3n+2}{4}} \sum _{i=1} ^{n-j} \left( \sum _{l=j+1} ^{n-i-1} \sum _{k=i+1} ^{n-l} 1  +\sum _{l=\frac{3n+4}{2}-j} ^{n} \sum _{k=\frac{3n+2}{2}-l} ^{j-1} 1 \right) \\
= & \frac{(n+2)(n-2)(5n^2-36n+84)}{1536}, \\
\sum _{ (i,j) \in N_3} \sigma (a_{(i,j)}) = 
& \sum _{j=\frac{3n+6}{4}} ^{n-1} \sum _{i=1} ^{n-j} \left( \sum _{l=j+1} ^{n-i-1} \sum _{k=i+1} ^{n-l} 1 + \sum _{l=j+1} ^{n} \sum _{k=\frac{3n+2}{2}-l} ^{j-1} 1  \right) \\
= & \frac{(n-2)(n-6)(n^2-4n+12)}{768}, \\
\sum _{ (i,j) \in N_4} \sigma (a_{(i,j)}) = 
& \sum _{j=\frac{3n+6}{4}} ^{n-1} \sum _{i=\frac{3n+4}{2}-j} ^{j-2} \sum _{l=j+1} ^{n} \sum _{k=i+1} ^{j-1} 1 
=\frac{(n-2)(n+2)(n-6)^2}{1536}, 
\end{align*}
\begin{align*}
\sum _{ (i,j) \in S_1} \sigma (a_{(i,j)}) = 
& \sum _{j=3} ^{\frac{n-2}{4}} \sum _{i=1} ^{j-2} \left( +\sum _{k=i+1} ^{j-1} \sum _{l=j+1} ^{\frac{n}{2}-k} 1 + \sum _{k=i+1} ^{j-1} \sum _{l=n-k+1} ^{n} 1 \right) 
=\frac{(n-2)(n-6)(n-10)(5n-6)}{6144}, \\
\sum _{ (i,j) \in S_2} \sigma (a_{(i,j)}) = 
& \sum _{i=1} ^{\frac{n-6}{4}} \sum _{j=\frac{n+2}{4}} ^{\frac{n}{2}-i} \left( \sum _{k=i+1} ^{\frac{n-2}{2}-j} \sum _{l=j+1} ^{\frac{n}{2}-k} 1+\sum _{k=i+1} ^{j-1} \sum _{l=n-k+1} ^{n} 1 \right)  
=\frac{(n+2)(n-2)(n-6)(11n+2)}{6144}, \\
\sum _{ (i,j) \in S_3} \sigma (a_{(i,j)}) = 
& \sum _{j=\frac{n+4}{2}} ^{\frac{3n-2}{4}} \sum _{i=n-j+1} ^{\frac{n}{2}}  \left( \sum _{l=j+1} ^{n} \sum _{k=i+1} ^{\frac{n}{2}} 1+\sum _{k=\frac{n+2}{2}} ^{j-1} \sum _{l=j+1} ^{\frac{3n}{2}-k} 1 \right)  
=\frac{(n+2)(n-2)(n-6)(3n-2)}{1536}, \\
\sum _{ (i,j) \in S_4} \sigma (a_{(i,j)}) = 
& \sum _{j=\frac{n+6}{2}} ^{\frac{3n-2}{4}} \sum _{i=\frac{n+2}{2}} ^{j-2} \sum _{k=i+1} ^{j-1} \sum _{l=j+1} ^{\frac{3n}{2}-k} 1 
=\frac{(n-2)(n-6)(n-10)(3n-10)}{6144}, \\
\sum _{ (i,j) \in S_5} \sigma (a_{(i,j)}) = 
& \sum _{j=\frac{3n+2}{4}} ^{n} \sum _{i=n-j+1} ^{\frac{n}{2}} \sum _{l=j+1} ^{n} \sum _{k=i+1} ^{\frac{3n}{2}-l} 1 
=\frac{n(n+2)(n-2)(n-4)}{384}, \\
\sum _{ (i,j) \in S_6} \sigma (a_{(i,j)}) = & 0, \\
\sum _{ (i,j) \in S_7} \sigma (a_{(i,j)}) = 
& \sum _{i=\frac{n+2}{2}} ^{\frac{3n-6}{4}} \sum _{j=\frac{3n+2}{4}} ^{\frac{3n}{2}-i} \sum _{l=j+1} ^{\frac{3n-2}{2}-i} \sum _{k=i+1} ^{\frac{3n}{2}-l} 1 
=\frac{(n+2)(n-2)(n-6)(n-10)}{6144}.
\end{align*}
Hence 
\begin{align*}
\varepsilon (D; H) = & \sum _{(i,j) \in N} \sigma (a_{(i,j)}) + \sum _{(i,j) \in S} \sigma (a_{(i,j)}) \\
 = & \sum_{t=1}^{4} \sum _{(i,j) \in N_t} \sigma (a_{(i,j)}) + \sum_{t=1}^{7} \sum _{(i,j) \in S_t} \sigma (a_{(i,j)}) \\
 = & \frac{n^4-8n^3+20n^2-32}{128} + \frac{(n-2)(n^3-6n^2+8n-16)}{128} \\
 = & \frac{n(n-4)(n-2)^2}{64} = Z(n).
\end{align*}


\noindent  (iv) When $n \equiv 3 \pmod 4$: Let 
\begin{align*}
N_1 = & \left\{ (i,j)\in X^2 \ | \  \frac{n+3}{2}-j \le i \le j-2, \ \frac{n+9}{4} \le j \le \frac{n-1}{2}   \right\} , \\
N_2 = & \left\{ (i,j)\in X^2 \ | \  1 \le i \le \frac{n-3}{2}, \  j = \frac{n+1}{2}   \right\} , \\
N_3 = & \left\{ (i,j)\in X^2 \ | \  1 \le i \le n-j, \ \frac{n+5}{2} \le j \le \frac{3n-1}{4}  \right\} , \\
N_4 = & \left\{ (i,j)\in X^2 \ | \  1 \le i \le n-j, \ \frac{3n+3}{4} \le j \le n-1  \right\} , \\
N_5 = & \left\{ (i,j)\in X^2 \ | \  \frac{3n+1}{2}-j \le i \le j-2, \ \frac{3n+7}{4} \le j \le n  \right\} .
\end{align*}
Then $N=N_1 \cup N_2 \cup \dots \cup N_5$ and $N_s \cap N_t = \emptyset$ for distinct $s, t\in \{1,2,3,4,5\}$. Let
\begin{align*}
S_1 = & \left\{ (i,j)\in X^2 \ | \  1 \le i \le j-2, \ 3 \le j \le \frac{n+1}{4}    \right\} , \\
S_2 = & \left\{ (i,j)\in X^2 \ | \  \frac{n+5}{4} \le j \le \frac{n+1}{2}-i, \ 1 \le i \le \frac{n-3}{4}   \right\} , \\
S_3 = & \left\{ (i,j)\in X^2 \ | \  n-j+1 \le i \le \frac{n-1}{2}, \ \frac{n+3}{2} \le j \le \frac{3n-1}{4}  \right\} , \\
S_4 = & \left\{ (i,j)\in X^2 \ | \  \frac{n+1}{2} \le i \le j-2, \ \frac{n+5}{2} \le j \le \frac{3n-1}{4}  \right\} , \\
S_5 = & \left\{ (i,j)\in X^2 \ | \  n-j+1 \le i \le \frac{n-1}{2}, \ \frac{3n+3}{4} \le j \le n-1   \right\}, \\
S_6 = & \left\{ (i,j)\in X^2 \ | \  2 \le i \le \frac{n-1}{2},\ j=n   \right\}, \\
S_7 = & \left\{ (i,j)\in X^2 \ | \  \frac{3n+3}{4} \le j \le \frac{3n-1}{2}-i, \ \frac{n+1}{2} \le i \le \frac{3n-5}{4}  \right\} .
\end{align*}
Then $S=S_1 \cup S_2 \cup \dots \cup S_7$ and $S_s \cap S_t = \emptyset$ for distinct $s, t\in \{1,2,3,4,5,6,7\}$. 
(An example is shown in Figure \ref{m19}.) 
\begin{figure}[ht]
\begin{center}
\includegraphics[width=85mm]{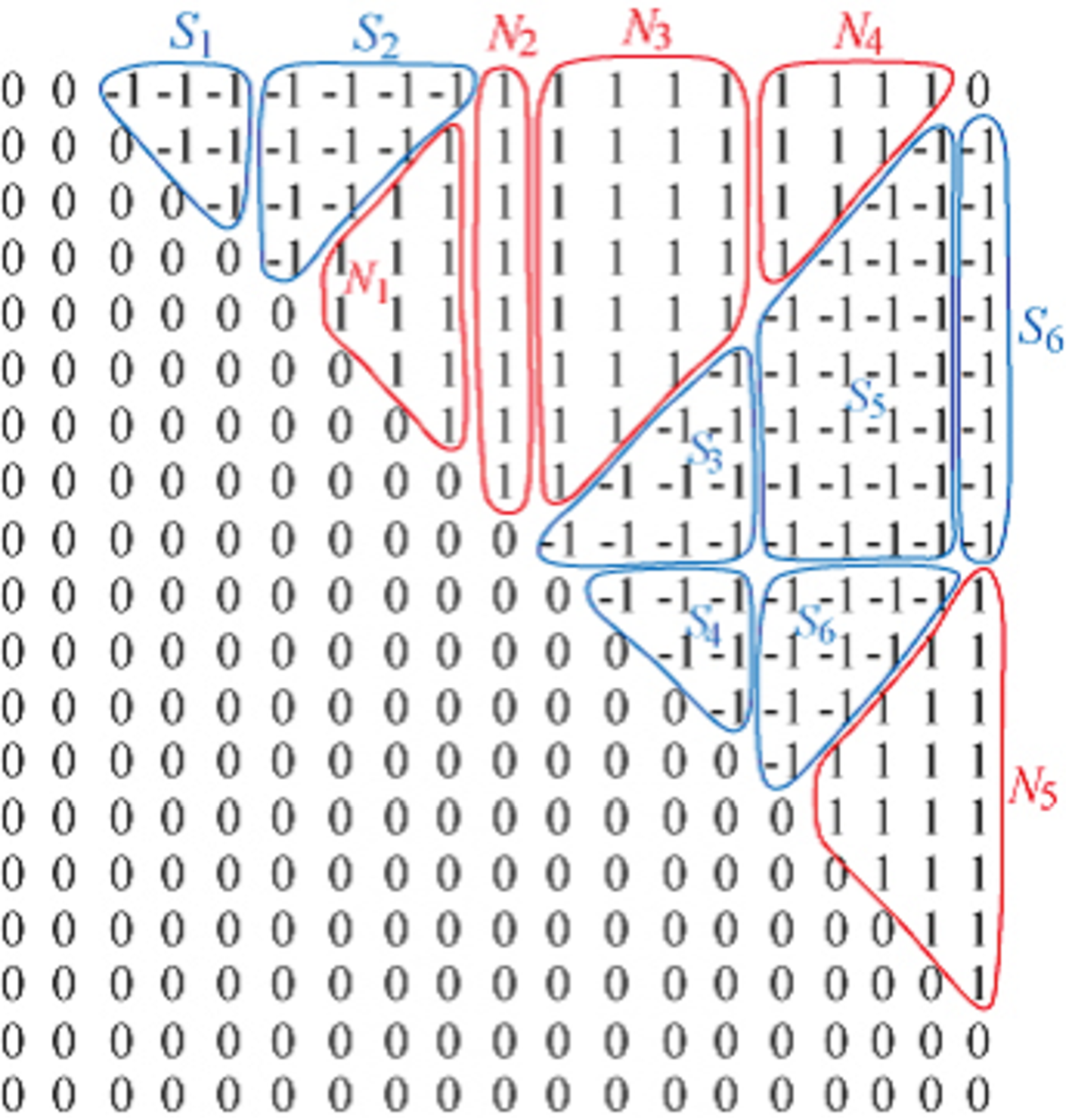}
\caption{Block division for $n=19$. }
\label{m19}
\end{center}
\end{figure}

\begin{align*}
\sum _{ (i,j) \in N_1} \sigma (a_{(i,j)}) = & \sum _{j=\frac{n+9}{4}} ^{\frac{n-1}{2}} \sum _{i=\frac{n+3}{2}-j} ^{j-2} \sum _{k=i+1} ^{j-1} \sum _{l=j+1} ^{n-k} 1 
=\frac{(n-3)(n-7)(n^2-4n+7)}{384}, \\
\sum _{ (i,j) \in N_2} \sigma (a_{(i,j)}) = & \sum _{i=1} ^{\frac{n-3}{2}} \sum _{k=i+1} ^{\frac{n-3}{2}} \sum _{l=\frac{n+3}{2}} ^{n-k} 1 
=\frac{(n-1)(n-3)(n-5)}{48}, \\
\sum _{ (i,j) \in N_3} \sigma (a_{(i,j)}) = & \sum _{j=\frac{n+3}{2}} ^{\frac{3n-1}{4}} \sum _{i=1} ^{n-j} \sum _{l=\frac{3n+3}{2}-j} ^{n} \sum _{k=\frac{3n+1}{2}-l} ^{j-1} 1 
=\frac{(n-3)(5n^3-41n^2+175n-163)}{1536}, \\
\sum _{ (i,j) \in N_4} \sigma (a_{(i,j)}) = & \sum _{j=\frac{3n+3}{4}} ^{n-1} \sum _{i=1} ^{n-j} \sum _{l=j+1} ^{n} \sum _{k=\frac{3n+1}{2}-l} ^{j-1} 1 
=\frac{(n+1)(n-3)(n^2-2n+13)}{768}, \\
\sum _{ (i,j) \in N_5} \sigma (a_{(i,j)}) = & \sum _{j=\frac{3n+7}{4}} ^{n} \sum _{i=\frac{3m+1}{2}-j} ^{j-2} \sum _{l=j+1} ^{n} \sum _{k=i+1} ^{j-1} 1 
=\frac{(n-3)(n-7)(n+1)^2}{1536}, 
\end{align*}
\begin{align*}
\sum _{ (i,j) \in S_1} \sigma (a_{(i,j)}) = 
& \sum _{j=3} ^{\frac{n+1}{4}} \sum _{i=1} ^{j-2} \left( \sum _{k=i+1} ^{j-1} \sum _{l=j+1} ^{\frac{n+1}{2}-k} 1 + \sum _{k=i+1} ^{j-1} \sum _{l=n-k+1} ^{n} 1 \right) 
=\frac{(n+1) (n-3) (n-7) (5n-7)}{6144}, \\
\sum _{ (i,j) \in S_2} \sigma (a_{(i,j)}) = 
& \sum _{i=1} ^{\frac{n-3}{4}} \sum _{j=\frac{n+5}{4}} ^{\frac{n+1}{2}-i} \left( \sum _{k=i+1} ^{\frac{n-1}{2}-j} \sum _{l=j+1} ^{\frac{n+1}{2}-k} 1+\sum _{k=i+1} ^{j-1} \sum _{l=n-k+1} ^{n} 1 \right)  
=\frac{(n-3)(n+1)(11n^2-22n-1)}{6144}, \\
\sum _{ (i,j) \in S_3} \sigma (a_{(i,j)}) = 
& \sum _{j=\frac{n+3}{2}} ^{\frac{3n-1}{4}} \sum _{i=n-j+1} ^{\frac{n-1}{2}}  \left( \sum _{l=j+1} ^{n} \sum _{k=i+1} ^{\frac{n-1}{2}} 1+\sum _{k=\frac{n+1}{2}} ^{j-1} \sum _{l=j+1} ^{\frac{3n-1}{2}-k} 1 \right)  
=\frac{(n+1)(n-3)^3}{512}, \\
\sum _{ (i,j) \in S_4} \sigma (a_{(i,j)}) = 
& \sum _{j=\frac{n+5}{2}} ^{\frac{3n-1}{4}} \sum _{i=\frac{n+1}{2}} ^{j-2} \sum _{k=i+1} ^{j-1} \sum _{l=j+1} ^{\frac{3n-1}{2}-k} 1 
=\frac{(n+1)(n-3)(n-7)(3n-17)}{6144}, \\
\sum _{ (i,j) \in S_5} \sigma (a_{(i,j)}) = 
& \sum _{j=\frac{3n+3}{4}} ^{n-1} \sum _{i=n-j+1} ^{\frac{n-1}{2}} \sum _{l=j+1} ^{n} \sum _{k=i+1} ^{\frac{3n-1}{2}-l} 1 
=\frac{(n+1)(n-1)(n-3)(n-5)}{384}, \\
\sum _{ (i,j) \in S_6} \sigma (a_{(i,j)}) = &  0, \\
\sum _{ (i,j) \in S_7} \sigma (a_{(i,j)}) = 
& \sum _{i=\frac{n+1}{2}} ^{\frac{3n-5}{4}} \sum _{j=\frac{3n+3}{4}} ^{\frac{3n-1}{2}-i} \sum _{l=j+1} ^{\frac{3n-3}{2}-i} \sum _{k=i+1} ^{\frac{3n-1}{2}-l} 1 
=\frac{(n+1)(n-3)(n-7)(n-11)}{6144}.
\end{align*}
Hence 
\begin{align*}
\varepsilon (D; H) = & \sum _{(i,j) \in N} \sigma (a_{(i,j)}) + \sum _{(i,j) \in S} \sigma (a_{(i,j)}) \\
 = & \sum_{t=1}^{5} \sum _{(i,j) \in N_t} \sigma (a_{(i,j)}) + \sum_{t=1}^{7} \sum _{(i,j) \in S_t} \sigma (a_{(i,j)}) \\
 = & \frac{(n-3)^2 (n^2-2n+5)}{128} + \frac{(n+1)(n-3)^3}{128} \\
 = & \frac{(n-1)^2(n-3)^2}{64} = Z(n).
\end{align*}

\hfill$\square$

\phantom{x}

\noindent Examples of based diagrams $(D; H)$ of $(K_n ; H)$ satisfying $\varepsilon (D; H) = Z(n)$ are shown in Figures \ref{K13} and \ref{K14} for $n=13$ and $14$. 
\begin{figure}[ht]
\begin{center}
\includegraphics[width=100mm]{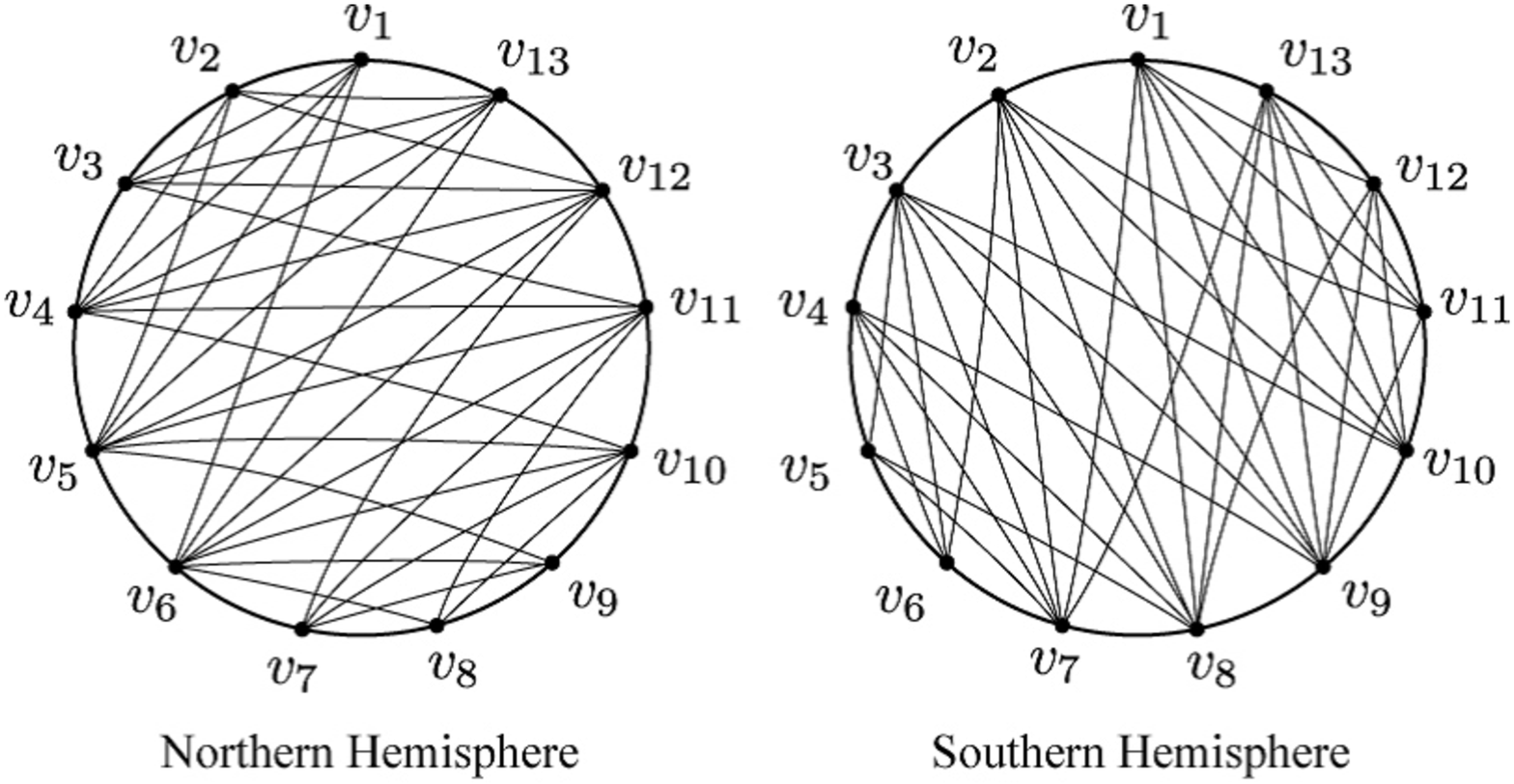}
\caption{An optimal based diagram $(D; H)$ of $(K_{13}; H)$ such that $M(D; H)=M_13$. }
\label{K13}
\end{center}
\end{figure}

\begin{figure}[ht]
\begin{center}
\includegraphics[width=100mm]{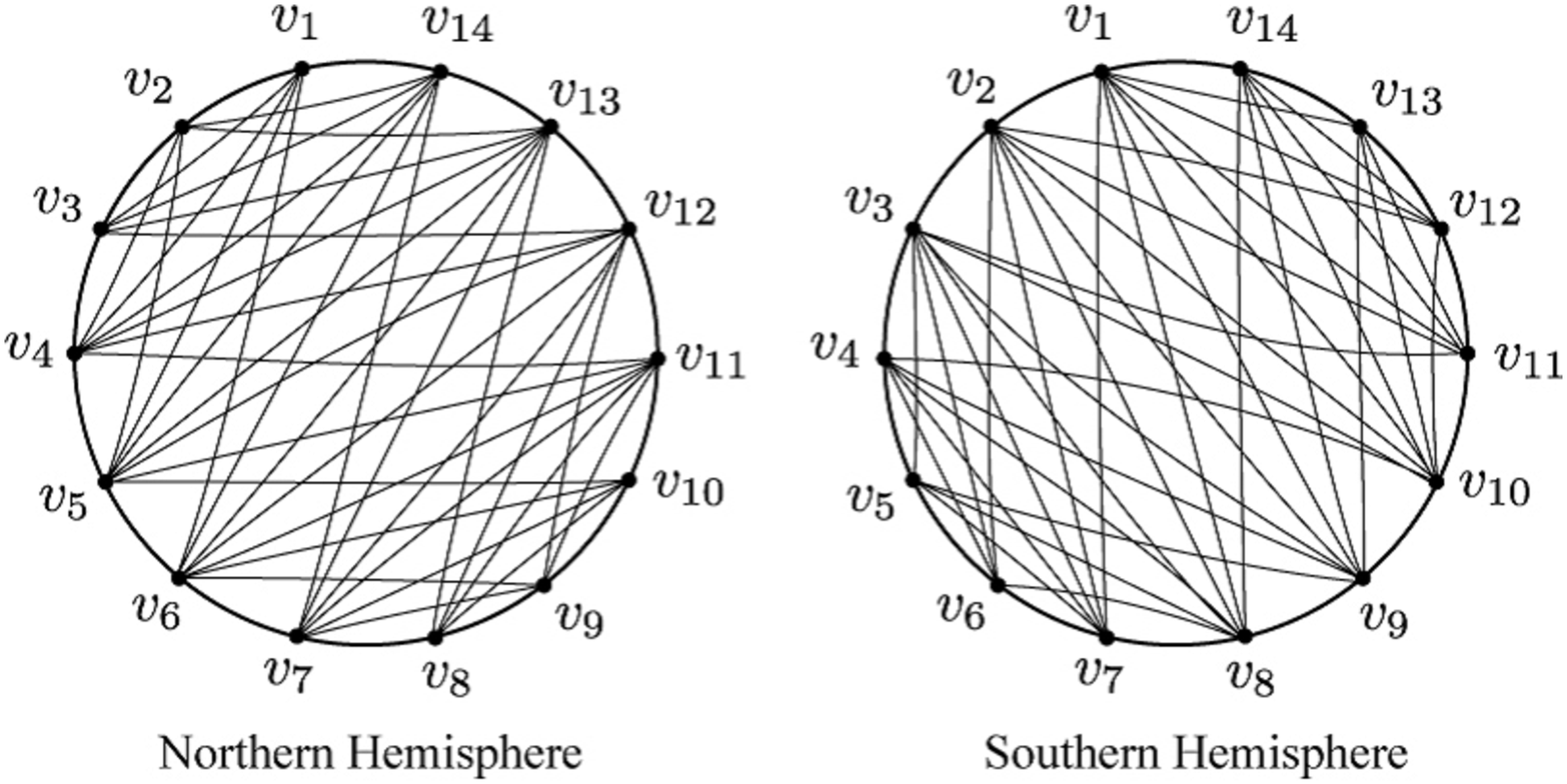}
\caption{An optimal based diagram $(D; H)$ of $(K_{14}; H)$ such that $M(D; H)=M_{14}$. }
\label{K14}
\end{center}
\end{figure}

\section{Proof of Theorem \ref{main-thm}}
In this section, the proof of Theorem \ref{main-thm} is given. 
Let $K_n$ be a complete graph with vertices $v_1, v_2, \dots ,v_n$. 
Let $H=v_1v_2\dots v_nv_1$ be a hamiltonian cycle of $K_n$. 
Let $(D; H)$ be a based diagram of $(K_n; H)$. 
We assume that the vertices $v_1, v_2, \dots ,v_n$ lay on $H$ ($\subset D$) clockwisely in this order. 
Let $E(D)$ and $E(H)$ be the set of all the edge diagrams in $D$ and $H$, respectively. 
For $e\in E(D)$, let $c_D(e)$ be the number of edge diagrams in $D$ which intersect with ${\rm Int}(e)$, where ${\rm Int}(e)$ is the open arc which is obtained by removing the end points of $e$. We remark that  $c_D(e)=0$ for $e\in E(H)$. 
The following lemmas are helpful for the proof of Theorem \ref{main-thm}.

\begin{lemma}([\cite{AAFRS1}, Theorem 21]) For any odd integer $n \ge 13$ and based diagram $(D; H)$ of $(K_n ; H)$, if $(D;H)$ is optimal, i.e., $c(D;H)=Z(n)$, then for any edge diagram $e\in E(D)\setminus E(H)$, we have $c_D(e)\ne 0$. 
\label{no-free}
\end{lemma}

\begin{lemma}For any odd integer $n\geq 5$ and based diagram $(D; H)$ of $(K_n ; H)$, if $M(D; H)=M_n$, then $c_D(e_{(m,n)})=c_D(e_{(1,m)})+c_D(e_{(2,n)})+1$, where $m=\frac{n+1}{2}$.
\label{remove-add} 
\end{lemma}

\begin{proof}
Put $M_n=(a_{(i,j)})$. 
Let $\widetilde{\sigma}(a_{(i,j)})$ be the number of components $a_{(k,l)}$ of $M_n$ satisfying $k<i<l<j$ and $a_{(i,j)}=a_{(k,l)}$. 
Then, $c_D(e_{(i,j)})=\sigma (a_{(i,j)})+\widetilde{\sigma}(a_{(i,j)})$ for any $1\leq i<j\leq n$. 
By the construction of $M_n$, $a_{(i,j)}=1$ for $1\leq i<j\leq n$ if and only if $j\geq i+2$ and $m+1\leq i+j\leq n$ or $j\geq i+2$ and $i+j\geq n+m$. 
Note that $a_{(1,m)}=1$, $a_{(2,n)}=-1$ and $a_{(m,n)}=1$. 
Since $\widetilde{\sigma}(a_{(1,m)})=0$ and $\sigma ({a_{(2,n)}})=\sigma (a_{(m,n)})=0$, we have 
$\displaystyle c_D(e_{(1,m)})=\sigma (a_{(1,m)})=\sum_{k=2}^{m-1}\sum_{l=m+1}^{n-k} 1$, 
$\displaystyle c_D(e_{(2,n)})=\widetilde{\sigma}(a_{(2,n)})=\sum_{l=3}^{m-1} 1=m-3$ and 
$\displaystyle c_D(e_{(m,n)})=\widetilde{\sigma}(a_{(m,n)})=\sum_{k=1}^{m-1}\sum_{l=m+1}^{n-k} 1=\sum_{k=2}^{m-1}\sum_{l=m+1}^{n-k} 1+\sum_{l=m+1}^{n-1} 1$. 
Since $\displaystyle \sum_{l=m+1}^{n-1} 1=n-m-1=m-2$, we have $c_D(e_{(m,n)})=c_D(e_{(1,m)})+c_D(e_{(2,n)})+1$.
\end{proof}

\noindent Theorem \ref{main-thm} is shown as follows: \\

\noindent {\bf Proof of Theorem \ref{main-thm}.} \ 
When $n=7$, we can see in Figure \ref{K7} that the theorem holds. We prove for $n\geq 9$. Take a based diagram $(D; H)$ of $(K_n; H)$ such that $M(D; H)=M_n$. Then, by Lemma \ref{mz-lem}, $(D; H)$ is optimal. Let $D'$ be a diagram of $K_n$ obtained from $D$ (see also Figures \ref{K9a} and \ref{K9b}) by removing the edge diagram $e_{(m,n)}$ (in the Northern Hemisphere) and adding a new edge diagram $e_{(m,n)}'$ for $m=\frac{n+1}{2}$ which satisfies; 
\begin{figure}[ht]
\begin{center}
\includegraphics[width=90mm]{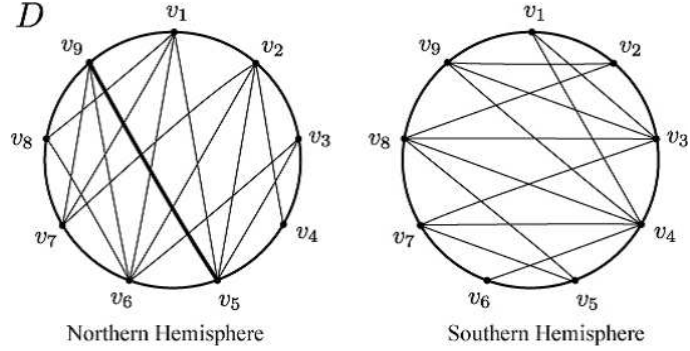}
\caption{An optimal based diagram $(D; H)$ of $(K_9; H)$ such that $M(D; H)=M_9$. }
\label{K9a}
\end{center}
\end{figure}

\begin{figure}[ht]
\begin{center}
\includegraphics[width=90mm]{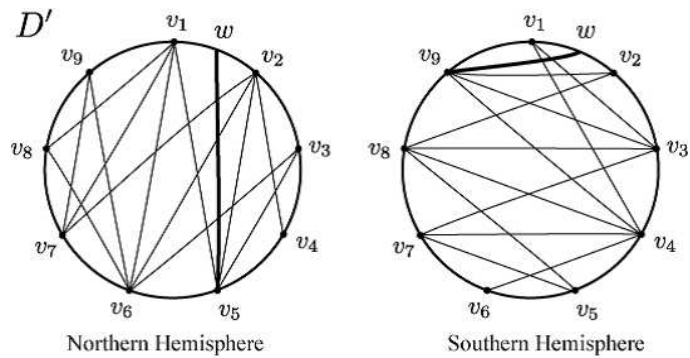}
\caption{An optimal diagram $D'$ of $K_9$ obtained from a based diagram $(D; H)$ of $(K_9; H)$ with $M(D; H)=M_9$ by removing $e_{(5,9)}$ and adding $e_{(5,9)}'$. }
\label{K9b}
\end{center}
\end{figure}

(i) the interior ${\rm Int}(e'_{(m,n)})$ of $e_{(m,n)}'$ intersects with $H$ at exactly one point $w$ in ${\rm Int}(e_{(1,2)})$,

(ii) the interior ${\rm Int}(\overline {v_m w})$ of the subarc $\overline {v_m w}$ of $e_{(m,n)}'$ connecting $v_m$ and $w$ is in the Northern Hemisphere and 

(iii) the interior ${\rm Int}(\overline {wv_n})$ of the subarc $\overline {wv_n}$ of $e_{(m,n)}'$ connecting $w$ and $v_n$ is in the Southern Hemisphere.

We prove $c(D)=c(D')$; 
Take an edge diagram $e_{(k,l)}$ of $D$ in the Northern Hemisphere with the conditions that $k<l$ and $(k,l)\ne (1,m), (m,n)$. 
Since $e_{(1,m)}$ and $e_{(k,l)}$ are also edge diagrams of $D'$ in the Northern Hemisphere, ${\rm Int}(e_{(1,m)})$ and ${\rm Int}(e_{(k,l)})$ intersect if and only if $1<k<m<l$, i.e., the four points $v_1$, $v_k$, $v_m$ and $v_l$ lay on $H$ clockwisely in this order. 
Since $w\in {\rm Int}(e_{(1,2)})$, the condition $1<k<m<l$ is equivalent to that the four points $w$, $v_k$, $v_m$ and $v_l$ lay on $H$ clockwisely in this order, i.e., ${\rm Int}(e_{(k,l)})$ and ${\rm Int}(\overline{v_m w})$ intersect. 
On the diagram $D'$, let $c_{D'}(e_{(m,n)}')$, $c_{D'}(\overline{v_m w})$ and $c_{D'}(\overline{w v_n})$ be the numbers of 
edge diagrams in $D'$ which intersect with ${\rm Int}(e_{(m,n)}')$, ${\rm Int}(c_{D'}(\overline{v_m w}))$ and ${\rm Int}(c_{D'}(\overline{wv_n}))$, respectively. Since ${\rm Int}(e_{(1,m)})\cap {\rm Int}(\overline{v_m w})=\emptyset$, we have 
$c_D(e_{(1,m)})=c_{D'}(\overline{v_m w})$. 
Similarly, we have $c_D(e_{(2,n)})=c_{D'}(\overline{w v_n})$. 
Since $|{\rm Int}(e_{(m,n)}')\cap H|=|\{w\}|=1$, we have 
$c_{D'}(e_{(m,n)}')=c_{D'}(\overline{v_m w})+c_{D'}(\overline{w v_n})+1=c_D(e_{(1,m)})+c_D(e_{(2,n)})+1$. 
By Lemma \ref{no-free}, we have $c_{D'}(e_{(m,n)}')=c_D(e_{(m,n)})$. Thus, $c(D)=c(D')=Z(n)$ holds. 
Let $T^L$ be the linear tree $v_2v_3\dots v_nv_1$ in the diagram $D'$. 
Since there are no crossings on $T^L$, 
$(D'; T^L)$ is a based diagram of $(K_n; T^L)$. 
When $n=9$ and $11$, we can see in Figures \ref{K9b} and \ref{K11} that $D'$ have no free hamiltonian cycle, respectvely. 
When $n\geq 13$, there is no free hamiltonian cycle except for $H$ in the diagram $D$ by Lemma \ref{no-free}, and hence $D'$ has no free hamiltonian cycle. 
This completes the proof of Theorem \ref{main-thm}. 
\hfill$\square$

\begin{figure}[ht]
\begin{center}
\includegraphics[width=100mm]{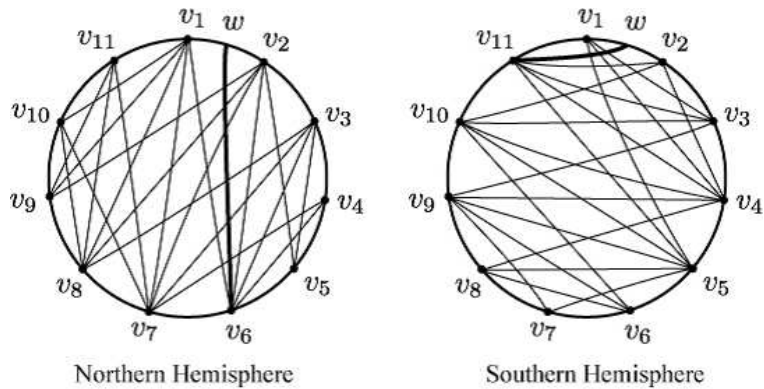}
\caption{An optimal diagram $D'$ of $K_{11}$ obtained from a based diagram $(D; H)$ of $(K_{11}; H)$ with $M(D; H)=M_{11}$ by removing $e_{(6,11)}$ and adding $e_{(6,11)}'$. }
\label{K11}
\end{center}
\end{figure}

\begin{remark}When $n=7$, the same way of construction above (the case of $n\geq 9$) does not work. See Figure \ref{rem7}.
\begin{figure}[ht]
\begin{center}
\includegraphics[width=40mm]{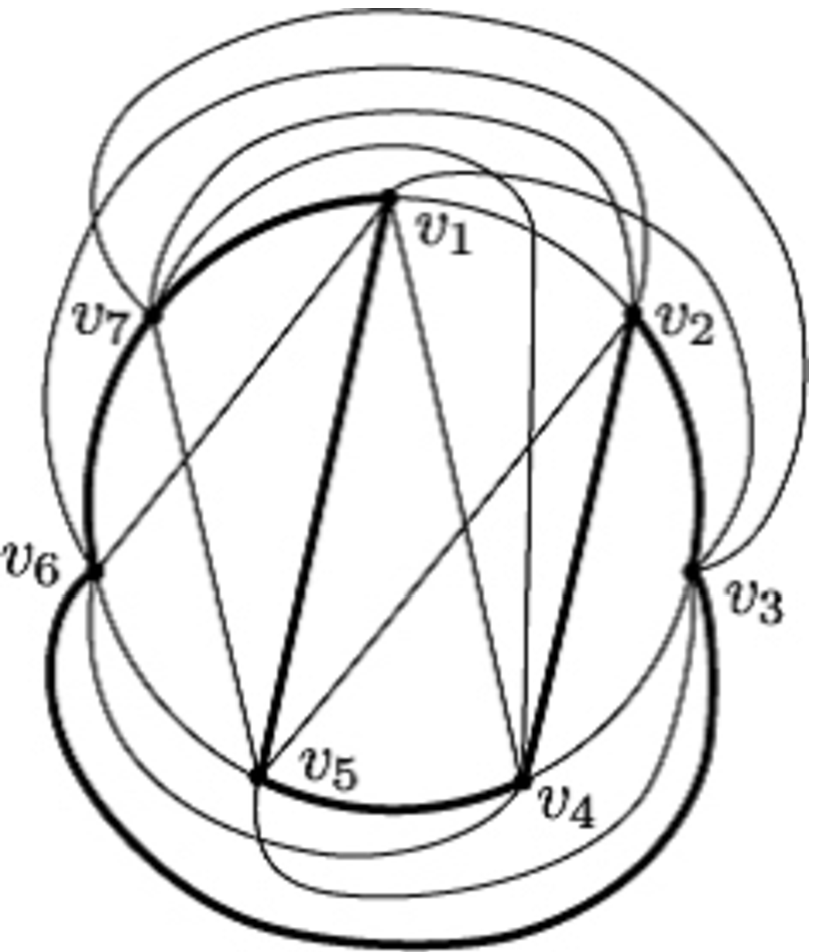}
\caption{An optimal diagram $D'$ of $K_7$ obtained from a based diagram $(D; H)$ of $(K_7; H)$ with $M(D; H)=M_7$ by removing $e_{(4,7)}$ and adding $e_{(4,7)}'$. 
The result of this construction has a hamiltonian cycle. }
\label{rem7}
\end{center}
\end{figure}

\end{remark}

\section*{Acknowledgments}
A. S. and Y. Y. thank Akio Kawauchi for valuable advice and encouragements. 
The revised virsion of this paper was written during A.S. and Y.Y.'s stay at Pusan National University. 
They would like to thank Sang Youl Lee and Jieon Kim for their kind hospitality. 
A. S. was partially supported by Grant for Basic Science Research Projects from The Sumitomo Foundation (160154).

\end{document}